\documentclass[11pt]{amsart}

\usepackage{amsmath,amssymb,amsthm,mathtools}
\usepackage{enumitem}
\usepackage{aliascnt}
\usepackage[hidelinks]{hyperref}
\usepackage[nameinlink,capitalise]{cleveref}

\newtheorem{theorem}{Theorem}[section]
\newaliascnt{proposition}{theorem}
\newtheorem{proposition}[proposition]{Proposition}
\aliascntresetthe{proposition}
\newaliascnt{lemma}{theorem}
\newtheorem{lemma}[lemma]{Lemma}
\aliascntresetthe{lemma}
\newaliascnt{corollary}{theorem}
\newtheorem{corollary}[corollary]{Corollary}
\aliascntresetthe{corollary}
\theoremstyle{definition}
\newaliascnt{definition}{theorem}
\newtheorem{definition}[definition]{Definition}
\aliascntresetthe{definition}
\newaliascnt{example}{theorem}
\newtheorem{example}[example]{Example}
\aliascntresetthe{example}
\theoremstyle{remark}
\newaliascnt{remark}{theorem}
\newtheorem{remark}[remark]{Remark}
\aliascntresetthe{remark}

\newcommand{\LLY}{\mathrm{LLY}}
\newcommand{\DN}{\mathrm{DN}}
\newcommand{\Curv}{\mathrm{Curv}}
\newcommand{\Etl}{E_{\mathrm{tl}}}
\newcommand{\core}{\operatorname{core}_2}
\newcommand{\Dist}{\operatorname{dist}}

\title[Curvature-Distortion Numbers of Graphs]{Curvature-Distortion Numbers of Graphs: Nonnegative Lin--Lu--Yau Curvature}
\author{Qing Xia}
\address{School of Mathematical Sciences\\
University of Science and Technology of China\\
96 Jinzhai Road\\
Hefei 230026, Anhui Province\\
China}
\email{xq0420@mail.ustc.edu.cn}
\date{September 2026}

\subjclass[2020]{05C05, 05C12, 05C63, 05C81}
\keywords{Lin--Lu--Yau curvature, curvature-distortion number, weighted graph, tree, graph end, branch vertex, 2-core, tree-like skeleton}

\begin{document}

\begin{abstract}
We introduce the curvature-distortion number, a scale-invariant
parameter measuring the least multiplicative spread of positive edge
weights required to make a weighted discrete curvature nonnegative
everywhere. We develop the theory for Lin--Lu--Yau curvature when the
transport distance is the fixed combinatorial graph distance.

For trees, the invariant admits an explicit nonlinear fixed-point
description, which produces a canonical optimal weight that is unique
up to scaling. More importantly, the curvature-distortion number
controls the branching topology of the tree: for every finite tree
$T$,
\[
|B(T)|\le \left\lceil \DN_{\LLY}(T)\right\rceil,
\]
where $B(T)$ is the set of branch vertices. Thus the amount of weight
distortion required to achieve nonnegative curvature imposes a direct
quantitative restriction on the topological complexity of the tree.
For locally finite infinite trees, finite distortion is
classified completely: it occurs precisely for the double ray and for
one-ended trees obtained from a finite tree by attaching a single ray.

This connection between curvature distortion and tree topology extends
naturally to general connected graphs through the subgraph formed by
edges lying in no cycle of length $3$, $4$, or $5$. Whenever the
curvature-distortion number is finite, this tree-like part is a forest
unless the whole graph is a cycle of length at least $6$, and each of
its tree components inherits the corresponding distortion and
topological bounds. In particular, the tree theory yields complete
finite-distortion classifications for graphs of girth at least $6$.
\end{abstract}

\maketitle

\section{Introduction and statement of results}
\label{sec:intro}

Discrete Ricci curvature on graphs originates with Ollivier's coarse
Ricci curvature \cite{Ollivier2009} and the high-idleness curvature of Lin,
Lu, and Yau \cite{LinLuYau2011}; see also \cite{BCLMP2018}. 

We work in the weighted graph-Laplacian framework of M\"unch and
Wojciechowski \cite{MunchWojciechowski2019}, specialized so that the
edge weights determine the normalized transition probabilities while
the transport metric remains the combinatorial graph distance. Lin and Liu use the same viewpoint for prescribed
curvature flows \cite{LinLiu2026}. Related weighted curvature flows and
discrete Einstein metrics on trees are studied in
\cite{BaiLiuLai2026,Cheng2026}. Our problem is instead variational: among all
positive weights producing nonnegative curvature, we minimize their
multiplicative spread.

Let $G=(V,E)$ be connected and locally finite, with $E\ne\varnothing$, and write $d_x$ for the combinatorial degree of $x$. Let $\kappa_{\LLY}^w$ denote the weighted curvature defined in \cref{subsec:weighted-LLY}. For a positive edge weight $w:E\to(0,\infty)$, define the \emph{distortion} of $w$ by
\[
\Dist(w):=\frac{\sup_{e\in E}w_e}{\inf_{e\in E}w_e},
\]
with value $+\infty$ if the numerator is infinite or the denominator is zero, and define the \emph{$LLY$-distortion number} of $G$ by
\begin{equation}
\label{eq:DN-intro}
\DN_{\LLY}(G)
:=\inf\left\{
\Dist(w):
\kappa^w_{\LLY}(e)\ge0\ \text{for every }e\in E
\right\}.
\end{equation}
We write $\kappa_{\LLY}^{w}(G)\ge0$ if
$\kappa_{\LLY}^{w}(e)\ge0$ for every $e\in E(G)$. A positive weight satisfying $\kappa_{\LLY}^{w}(G)\ge0$ is called \emph{admissible}. An admissible weight is called
\emph{optimal} if its distortion equals $\DN_{\LLY}(G)$. A
finite-distortion weight is called \emph{normalized} if
$\inf_{e\in E}w_e=1$. For finite $E$, write
$w_{\max}:=\max_{e\in E}w_e$ and $w_{\min}:=\min_{e\in E}w_e$; then
normalization means $w_{\min}=1$, and \eqref{eq:DN-intro} is the infimum
of $w_{\max}/w_{\min}$ over admissible weights. If no admissible weight has finite distortion, we set $\DN_{\LLY}(G)=\infty$; by convention, $\DN_{\LLY}(K_1)=1$. A general curvature-distortion number is defined in \cref{def:general-DN}.

For an edge $e\in E$, write $e\in\Etl(G)$ if $e$ lies on no cycle of length $3$, $4$, or $5$, and call such an edge \emph{tree-like}. Define the \emph{tree-like skeleton} of $G$ by
\[
\mathcal S(G):=(V,\Etl(G)).
\]
A component of $\mathcal S(G)$ is called \emph{nontrivial} if it contains an edge. Thus $\Etl(G)=E$ whenever $G$ has girth at least $6$. Our first result gives a necessary condition for finite distortion.

\begin{theorem}[Forest obstruction]
\label{thm:necessary}
Let $G=(V,E)$ be a connected locally finite graph. If $\DN_{\LLY}(G)<\infty$, then either
\[
G=C_n\qquad\text{for some }n\ge6,
\]
or $\mathcal S(G)$ is a forest.
\end{theorem}

The converse fails even for finite graphs: \cref{ex:forest-not-sufficient} gives a seven-vertex graph for which $\mathcal S(G)$ is a forest but $\DN_{\LLY}(G)=\infty$.

For finite graphs, the hanging trees outside the $2$-core can be compressed without changing finiteness of the distortion number; see \cref{def:2-core}. Let $K=\core(G)\ne\varnothing$ and
\[
A_G:=\{x\in V(K):x\text{ has a neighbor in }V(G)\setminus V(K)\}.
\]
Let $\widehat K_G$ be obtained from $K$ by attaching one pendant vertex $x^\partial$ to each $x\in A_G$. Thus $\widehat K_G$ records where nontrivial hanging forests are attached, but not their sizes or shapes.

\begin{theorem}[$2$-core reduction]
\label{thm:core-reduction}
Let $G$ be a finite connected graph with
\[ K:=\core(G)\ne\varnothing. \]
Then
\[
\DN_{\LLY}(G)<\infty
\quad\Longleftrightarrow\quad
\DN_{\LLY}(\widehat K_G)<\infty.
\]
\end{theorem}

Thus finiteness is determined by the $2$-core together with the core
vertices carrying nontrivial hanging forests, although the numerical value
need not be preserved. See \cref{cor:core-extension,cor:girth6} for useful
consequences.

We next turn to trees. An edge is \emph{pendant} if it is incident with a leaf, and $E_{\mathrm{np}}(T)$ denotes the set of non-pendant edges. The nonlinear map $F_T$ is defined in \cref{def:distortion-map}. The next theorem identifies $\DN_{\LLY}(T)$ through its unique fixed point and the associated normalized optimal weight.

\begin{theorem}[Finite-tree fixed-point formula]
\label{thm:tree-fixed-point}
Let $T$ be a finite tree with $E(T)\ne\varnothing$. Then $F_T$ has a
unique fixed point
\[
w_T^\star\in[1,\infty)^{E(T)}.
\]
Moreover,
\[
\DN_{\LLY}(T)=\max_{e\in E(T)}w_T^\star(e).
\]
The fixed point is coordinatewise minimal among normalized admissible
weights and is the unique normalized optimal weight. Equivalently, all
optimal weights differ only by a global positive scalar. We call
$w_T^\star$ the \emph{canonical weight} of $T$.

Every pendant edge has canonical weight $1$, while every non-pendant
edge has zero Lin--Lu--Yau curvature under $w_T^\star$.
\end{theorem}

Let $T$ be a finite tree. If $v$ is a vertex of degree $2$ with
neighbors $x$ and $y$, suppressing $v$ means deleting $v$ and the
edges $xv,vy$, and adding the edge $xy$. The \emph{$2$-suppression}
$\overline T$ of $T$ is obtained by repeatedly suppressing degree-two
vertices until none remain. Suppression preserves the branch vertices in the natural correspondence, and hence
\begin{equation}
\label{eq:branch-preserved}
|B(\overline T)|=|B(T)|.
\end{equation}
The following theorem shows that suppressing degree-two vertices does not increase the invariant.

\begin{theorem}[Suppression monotonicity]
\label{thm:suppression-intro}
For every finite tree $T$,
\[
\DN_{\LLY}(\overline T)\le\DN_{\LLY}(T).
\]
\end{theorem}

The inequality can be strict. 
A vertex $v$ is called a \emph{branch vertex} if $d_v\ge3$. Writing
\[
B(T):=\{v\in V(T):d_v\ge3\},
\]
the canonical weight also yields a sharp global restriction on the number of branch vertices.

\begin{theorem}[Branch-vertex bound]
\label{thm:branch-intro}
Every finite tree $T$ satisfies
\[
|B(T)|\le \left\lceil \DN_{\LLY}(T)\right\rceil.
\]
\end{theorem}

The fixed-point theory extends to locally finite infinite trees.

\begin{theorem}[Infinite-tree fixed points and exhaustion]
\label{thm:infinite-fixed-exhaustion-intro}
Let $T$ be a locally finite infinite tree. Then
\[
\DN_{\LLY}(T)<\infty
\]
if and only if $F_T$ admits a bounded normalized fixed point. Whenever
these conditions hold, the bounded normalized fixed point is unique;
denote it by $w_T^\star$. It is the unique normalized optimal weight,
and
\begin{equation}
\label{eq:infinite-fixed-formula-intro}
\DN_{\LLY}(T)=\sup_{e\in E(T)}w_T^\star(e).
\end{equation}
We call $w_T^\star$ the \emph{canonical weight} of $T$. Every
non-pendant edge has zero Lin--Lu--Yau curvature under $w_T^\star$,
while every pendant edge has weight $1$.
Moreover, for every exhaustion
\[
T_1\subset T_2\subset\cdots,
\qquad
\bigcup_{n\ge1}T_n=T,
\]
by finite subtrees,
\begin{equation}
\label{eq:exhaustion-intro}
\DN_{\LLY}(T_n)\nearrow\DN_{\LLY}(T),
\end{equation}
and the canonical weights $w_{T_n}^\star$ converge pointwise on
$E(T)$ to $w_T^\star$.
\end{theorem}

By \cref{thm:infinite-fixed-exhaustion-intro}, the branch-vertex bound
for finite trees extends to locally finite infinite trees; in particular,
every infinite tree of finite distortion has only finitely many branch
vertices. We then obtain the following classification. The notions of ray, double ray, and end are recalled in \cref{subsec:infinite-ends}.

\begin{theorem}[Infinite-tree classification]
\label{thm:ends-intro}
Let $T$ be a locally finite infinite tree. Then
\[
\DN_{\LLY}(T)<\infty
\]
if and only if exactly one of the following holds:
\begin{enumerate}[label=\textup{(\roman*)}]
\item $T$ is a double ray;
\item $T$ is obtained from a finite tree by attaching one degree-two ray.
\end{enumerate}
Equivalently, case \textup{(ii)} consists precisely of the one-ended trees with finitely many branch vertices. Moreover, whenever $D:=\DN_{\LLY}(T)<\infty$,
\begin{equation}
\label{eq:infinite-branch-bound-intro}
|B(T)|\le\lceil D\rceil.
\end{equation}
Thus every finite-distortion infinite tree has at most two ends, with two ends only in the double-ray case.
\end{theorem}

Finally, we return to the tree-like skeleton $\mathcal S(G)$. Its relevance to Lin--Lu--Yau curvature comes from the tree-like local formula on every edge of $\Etl(G)$. Since $\mathcal S(G)$ may be disconnected, distortion numbers are applied only to its nontrivial connected components.

\begin{theorem}[Tree-like subgraph comparison]
\label{thm:skeleton-intro}
Let $G$ be a connected locally finite graph. Every connected subgraph $H\subseteq G$ with $E(H)\subseteq\Etl(G)$ satisfies
\[
\DN_{\LLY}(H)\le\DN_{\LLY}(G).
\]
\end{theorem}

A complementary local estimate holds on every tree-like edge $xy$:
\[
\sqrt{(d_x-1)(d_y-1)}\le\DN_{\LLY}(G).
\]
Together with the tree theory, \cref{thm:necessary,thm:skeleton-intro}
control the distortion, branch vertices, ends, and local degrees of every
nontrivial component of $\mathcal S(G)$; see
\cref{cor:finite-distortion-skeleton-structure}.

\bigskip

\paragraph{Relation to prescribed curvature flow.} For finite graphs of girth at least $6$, Lin and Liu \cite[Proposition~3.2 and Theorems~3.2--3.3]{LinLiu2026} proved projective uniqueness for a prescribed attainable curvature and exponential convergence of the corresponding curvature flow. Hence, if $w$ is an optimal weight, then the nonnegative curvature vector
\[
\bigl(\kappa_{\LLY}^w(e)\bigr)_{e\in E}
\]
is realized uniquely up to a global positive rescaling of $w$, and
every realizing weight has distortion $\DN_{\LLY}(G)$. For a finite tree, the canonical weight satisfies \[ \kappa_{\LLY}^{w_T^\star}(e)=0 \] on every non-pendant edge.

\bigskip

\paragraph{Organization.}
Section~\ref{sec:prelim} collects the weighted LLY preliminaries.
Section~\ref{sec:finiteness} proves the forest obstruction and the exact
$2$-core reduction. Section~\ref{sec:finite-trees} develops the finite-tree
fixed-point theory, suppression monotonicity, and the branch-vertex bound.
Section~\ref{sec:infinite-trees} proves the exhaustion theorem and classifies
infinite trees of finite distortion. Section~\ref{sec:tree-like-skeleton}
then transfers the tree theory to tree-like subgraphs of general connected locally finite
graphs.

\section{Preliminaries}
\label{sec:prelim}

All graphs are simple and undirected. Graphs carrying $\DN_{\LLY}$ are
connected, and infinite graphs are locally finite; auxiliary subgraphs may
be disconnected and are treated componentwise. Unless stated otherwise,
graphs are finite. 

For a connected locally finite graph $G=(V,E)$, write $x\sim y$ for $xy\in E$,
$N_G(x):=\{y\in V:y\sim x\}$,
$N_G[x]:=N_G(x)\cup\{x\}$, $d_x$ for
the combinatorial degree, and $d_G(x,y)$ for graph distance. If the
ambient graph is clear, we write $d(x,y)$. The girth of a forest is
understood to be infinite.

\subsection{Weighted Lin--Lu--Yau curvature}
\label{subsec:weighted-LLY}

Let $w:E\to(0,\infty)$ be a positive edge weight. The weighted degree of a vertex $x$ is
\[
d_x^w:=\sum_{y\sim x}w_{xy}.
\]
When needed, write $d_{x,G}^w$. For $E\ne\varnothing$, set
\[
w_{\inf}:=\inf_{e\in E}w_e,
\qquad
w_{\sup}:=\sup_{e\in E}w_e,
\qquad
\Dist(w):=\frac{w_{\sup}}{w_{\inf}},
\]
with value $\infty$ if $w_{\inf}=0$ or $w_{\sup}=\infty$. On finite
graphs we use $w_{\min}$ and $w_{\max}$. Throughout, the transport
distance is the combinatorial distance; only the transition
probabilities depend on $w$.

For $\alpha\in[0,1]$, define
\begin{equation}
\label{eq:mu}
\mu_x^{\alpha,w}(z)=
\begin{cases}
\alpha, & z=x,\\[2mm]
(1-\alpha)\dfrac{w_{xz}}{d_x^w}, & z\sim x,\\[3mm]
0, & \text{otherwise}.
\end{cases}
\end{equation}
For finitely supported probability measures $\mu,\nu$ on $V$, let $\Pi(\mu,\nu)$ denote the set of couplings of $\mu$ and $\nu$. Their $1$-Wasserstein distance with respect to $d$ is
\[
W_1(\mu,\nu)
:=\inf_{\pi\in\Pi(\mu,\nu)}
\sum_{u,v\in V}d(u,v)\pi(u,v).
\]
For adjacent vertices $x\sim y$, define
\[
\kappa_\alpha^w(x,y)
:=1-W_1\bigl(\mu_x^{\alpha,w},\mu_y^{\alpha,w}\bigr)
\]
and
\begin{equation}
\label{eq:LLY}
\kappa_{\LLY}^w(x,y)
:=\lim_{\alpha\to1^-}\frac{\kappa_\alpha^w(x,y)}{1-\alpha}.
\end{equation}
See \cite{LinLuYau2011,BCLMP2018} for the underlying idleness theory.
The fixed-distance weighted convention in \eqref{eq:mu} is the normalized
weighted-Laplacian framework of M\"unch and Wojciechowski
\cite{MunchWojciechowski2019}. We call $w\equiv1$ the \emph{unit weight}; in this case we omit the superscript $w$. When the ambient graph must be displayed, we write $\kappa_{\LLY}^{w,G}$ (and similarly for $\kappa_{\alpha}^{w,G}$).

For a function $f:V\to\mathbb R$, write
\[
\operatorname{Lip}(1)
:=\{f:V\to\mathbb R:|f(u)-f(v)|\le d(u,v)\text{ for all }u,v\in V\}.
\]
We also write
\[
\Delta_w f(x)
:=\frac{1}{d_x^w}\sum_{z\sim x}w_{xz}\bigl(f(z)-f(x)\bigr)
\]
whenever $d_x^w>0$. We shall repeatedly use the following limit-free formulation of
Lin--Lu--Yau curvature, due to M\"unch and Wojciechowski
\cite[Theorem~2.1 and Corollary~2.2]{MunchWojciechowski2019}. In the
present normalized weighted setting it reads
\begin{equation}
\label{eq:dual-LLY}
\kappa_{\LLY}^w(x,y)
=
\inf\left\{
\Delta_w f(x)-\Delta_w f(y):
f\in\operatorname{Lip}(1),\ f(y)-f(x)=1
\right\}.
\end{equation}

\begin{lemma}[Scale invariance]
\label{lem:scale}
For every $c>0$,
\[
\kappa_{\LLY}^{cw}(e)=\kappa_{\LLY}^{w}(e)
\qquad\text{for every }e\in E.
\]
\end{lemma}

\begin{proof}
The ratios $w_{xy}/d_x^w$ in \eqref{eq:mu} are unchanged by the global rescaling $w\mapsto cw$. Hence all measures in \eqref{eq:mu}, and therefore all curvatures, are unchanged.
\end{proof}

By \cite[Lemma~2.2]{LinLiu2026}, we have the following regularity
property.
\begin{lemma}[Continuity]\label{lem:continuity}
For a fixed finite graph $G$, the map
\[
(0,\infty)^E\longrightarrow\mathbb R^E,
\qquad
w\longmapsto\bigl(\kappa_{\LLY}^w(e)\bigr)_{e\in E},
\]
is locally Lipschitz, and hence continuous.
\end{lemma}

M\"unch and Wojciechowski compute the curvature of an edge contained
in no $3$-, $4$-, or $5$-cycle in
\cite[Example~2.3]{MunchWojciechowski2019}. Specializing their formula to
the normalized weighted Laplacian used here gives the following local
identity.

\begin{lemma}[Tree-like edge formula]
\label{lem:tree-like-formula}
Let $xy\in E(G)$ be contained in no cycle of length $3$, $4$, or $5$. Then
\begin{equation}
\label{eq:tree-like-curvature}
\kappa_{\LLY}^w(x,y)
=2w_{xy}\left(\frac{1}{d_x^w}+\frac{1}{d_y^w}\right)-2.
\end{equation}
Consequently,
\begin{equation}
\label{eq:local-nonnegative}
\kappa_{\LLY}^w(x,y)\ge0
\quad\Longleftrightarrow\quad
w_{xy}^2\ge
\bigl(d_x^w-w_{xy}\bigr)
\bigl(d_y^w-w_{xy}\bigr).
\end{equation}
Equality in \eqref{eq:local-nonnegative} is equivalent to $\kappa_{\LLY}^w(x,y)=0$.
\end{lemma}

For a tree, \cref{lem:tree-like-formula} applies to every edge. In particular, if $x$ is a leaf and $xy$ is its unique incident edge, then
\begin{equation}
\label{eq:leaf-curvature}
\kappa_{\LLY}^w(x,y)=\frac{2w_{xy}}{d_y^w}>0.
\end{equation}

\subsection{Curvature-distortion numbers}

We retain the notation and terminology introduced in
\eqref{eq:DN-intro}. Thus $\DN_{\LLY}(G)$ is defined only for connected
locally finite graphs, with the convention
\[
\DN_{\LLY}(K_1)=1.
\]
No distortion number is assigned to a disconnected graph in this
paper.

\begin{definition}
\label{def:general-DN}
Let $G$ be connected and locally finite, and let $\Curv$ be a weighted discrete curvature indexed by a local index set $\mathcal I(G)$. Assume that its sign is invariant under global rescaling of positive edge weights. Define the \emph{$\Curv$-distortion number} by
\[
\DN_{\Curv}(G)
:=\inf\left\{
\Dist(w):
 w:E\to(0,\infty),\quad
 \Curv^w(\xi)\ge0\ \text{for every }\xi\in\mathcal I(G)
\right\},
\]
with value $+\infty$ if no positive weight satisfying these curvature inequalities has finite multiplicative distortion. If $E=\varnothing$, connectedness forces $G=K_1$, and we set $\DN_{\Curv}(K_1)=1$ by convention.
\end{definition}

In the remainder of the paper, $\mathcal I(G)=E$ and $\Curv=\LLY$. By \cref{lem:scale}, every finite-distortion admissible weight can be normalized by $w_{\inf}=1$ without changing its curvature or distortion. On a finite graph this is the usual normalization $w_{\min}=1$.

\begin{proposition}[Attainment on finite graphs]
\label{prop:finite-attainment}
Let $G$ be a finite connected graph. If
\[
\DN_{\LLY}(G)<\infty,
\]
then $G$ admits an optimal weight.
\end{proposition}

\begin{proof}
The case $G=K_1$ is immediate. Assume $E(G)\ne\varnothing$.
Choose a normalized admissible weight $w_0$ and set
\[
M:=\Dist(w_0).
\]
Consider
\[
\mathcal W_M(G)
:=
\left\{
w\in[1,M]^E:
\min_{e\in E}w_e=1,\ 
\kappa_{\LLY}^w(e)\ge0\ \text{for every }e\in E
\right\}.
\]
This set is nonempty, since $w_0\in\mathcal W_M(G)$, and it is compact
by \cref{lem:continuity}. Hence the continuous function
\[
w\longmapsto\Dist(w)=\max_{e\in E}w_e
\]
attains its minimum on $\mathcal W_M(G)$.

This minimum equals $\DN_{\LLY}(G)$. Indeed, if the minimum were
strictly larger,
the definition of $\DN_{\LLY}(G)$ would give a normalized admissible
weight with smaller distortion; since this distortion is then less than
$M$, that weight would also belong to $\mathcal W_M(G)$, a contradiction.
Thus a minimizer in $\mathcal W_M(G)$ is an optimal weight.
\end{proof}

For a finite connected graph $G$ with at least one edge,
\[
\DN_{\LLY}(G)=1
\]
if and only if $\kappa_{\LLY}(G)\geq 0$. Indeed, if the unit weight is admissible, then $\DN_{\LLY}(G)=1$. Conversely, it is clear from the definition.

\subsection{Trees and the LLY-distortion map}
\label{subsec:tree-distortion-map} 

Let $T$ be a locally finite tree. Put \[ \mathcal X_T:=[1,\infty)^{E(T)}. \] For $x=(x_e)_{e\in E(T)}\in\mathcal X_T$ and an edge $e=u\sim v$, define \begin{equation} \label{eq:Aue} A_{u\mid v}(x) := \sum_{\substack{f\in E(T)\\ f\ni u,\ f\ne e}}x_f. \end{equation} The sum is finite because $T$ is locally finite. 

\begin{definition}[LLY-distortion map] \label{def:distortion-map} For a locally finite tree $T$, define \[ F_T:\mathcal X_T\longrightarrow\mathcal X_T \] coordinatewise by \begin{equation} \label{eq:distortion-map} (F_T(x))_{uv} := \max\left\{ 1,\sqrt{A_{u\mid v}(x)A_{v\mid u}(x)} \right\}, \qquad uv\in E(T). \end{equation} A vector $x\in\mathcal X_T$ is called a \emph{fixed point} of $F_T$ if 
\begin{equation}\label{eq:distortion-system}
    F_T(x)=x,
\end{equation}
and \emph{feasible} if $F_T(x)\le x$ coordinatewise. It is called \emph{bounded} if \[ \sup_{e\in E(T)}x_e<\infty, \] and \emph{normalized} if \[ \inf_{e\in E(T)}x_e=1. \] 
A fixed point $x$ of $F_T$ is called an \emph{optimal fixed point} if, when regarded as an edge weight on $T$, it is an optimal weight.
\end{definition} 

When $T$ is finite, every element of $\mathcal X_T$ is bounded, and normalization is equivalent to \[ \min_{e\in E(T)}x_e=1. \]

By \eqref{eq:local-nonnegative}, for every edge $u\sim v$ of a tree, \[ \kappa_{\LLY}^w(u,v)\ge0 \quad\Longleftrightarrow\quad w_{uv}^2\ge A_{u\mid v}(w)A_{v\mid u}(w), \] with equality if and only if $\kappa_{\LLY}^w(u,v)=0$. Thus, when $x\in\mathcal X_T$ is regarded as an edge weight, $x$ is admissible if and only if it is feasible. Hence 
\begin{equation}\label{eq:supersolution}
    x\ \text{is a normalized admissible weight}\quad\Longleftrightarrow\quad\ x\ \text{is a feasible vector of \cref{def:distortion-map}}.
\end{equation}
Notice that if $e$
is pendant, then one of the two complementary loads in
\eqref{eq:distortion-map} is zero. Hence
\[
(F_T(x))_e=1.
\]
In particular, every fixed point of $F_T$ assigns weight $1$ to every
pendant edge.

We shall also use the following form of the zero-curvature equation. Let $w$ be a positive weight on a locally finite tree. If $u\sim v$ is a non-pendant edge with $\kappa_{\LLY}^w(u,v)=0$, then \[ w_{uv}^2=(d_u^w-w_{uv})(d_v^w-w_{uv}), \] and hence \begin{equation} \label{eq:harmonic-edge} w_{uv}=\frac{d_u^wd_v^w}{d_u^w+d_v^w}. \end{equation}

For a positive weight under which every non-pendant edge has zero curvature, orient a non-pendant edge $u\sim v$ from $u$ toward $v$ when $d_u^w<d_v^w$. If the two weighted degrees are equal, leave the edge unoriented and call it a \emph{tie edge}. 

\begin{lemma}[One-outgoing-edge property] \label{lem:one-outgoing} Let $T$ be a locally finite tree with a positive edge weight $w$ such that every non-pendant edge has zero curvature. If a non-pendant edge $u\sim v$ is oriented from $u$ toward $v$, then every other non-pendant edge incident with $u$ is oriented toward $u$. In particular, a vertex cannot be incident with two strict outgoing non-pendant edges, and a strict outgoing edge cannot coexist with a tie edge at the same vertex. 
\end{lemma} 

\begin{proof} If $d_u^w<d_v^w$, then \eqref{eq:harmonic-edge} gives $w_{uv}/d_u^w>1/2$. Hence every other incident edge $uz$ has $w_{uz}<d_u^w/2$, and another use of \eqref{eq:harmonic-edge} gives $d_z^w<d_u^w$. Thus every other non-pendant edge is oriented toward $u$. \end{proof}

Along two consecutive edges oriented in the same direction, the edge
weights strictly increase in that direction. Indeed, this follows from
\eqref{eq:harmonic-edge} because the function $ab/(a+b)$ is strictly
increasing in each variable. We use these observations in the proof of \cref{lem:one-suppression,lem:strong-rooted-load,lem:terminal-configuration,thm:ends-intro}.

\begin{lemma}[Monotone iteration]
\label{lem:monotone-iteration}
Let $T$ be a locally finite tree and let
$y\in[1,\infty)^{E(T)}$ be bounded and feasible. Define
\[
x^{(0)}:=y,
\qquad
x^{(n+1)}:=F_T(x^{(n)}).
\]
Then $(x^{(n)})_{n\ge0}$ decreases coordinatewise to a bounded fixed
point $x^\star$ satisfying
\[
x^\star\le y.
\]
If
\[
\inf_{e\in E(T)}y_e=1,
\]
then
\[
\inf_{e\in E(T)}x_e^\star=1.
\]
\end{lemma}

\begin{proof}
Since $F_T$ is coordinatewise nondecreasing and $F_T(y)\le y$,
\[
\mathbf1\le x^{(n+1)}\le x^{(n)}\le y.
\]
Thus $x^{(n)}$ converges coordinatewise to a bounded vector $x^\star$.
Local finiteness makes each coordinate of $F_T$ a finite expression, so
passing to the limit gives $F_T(x^\star)=x^\star$. If $\inf_e y_e=1$,
then $\mathbf1\le x^\star\le y$ also gives $\inf_e x_e^\star=1$.
\end{proof}

\subsection{The \texorpdfstring{$2$}{2}-core and cycle rank}

\begin{definition}\label{def:2-core}
For a finite graph $G$, the \emph{$2$-core} $\core(G)$ is the unique maximal subgraph of $G$ with minimum degree at least $2$. Equivalently, it is obtained by repeatedly deleting vertices of degree $0$ or $1$ until no such vertex remains.
\end{definition}

If $H$ is a finite graph with $c(H)$ connected components, its \emph{cycle rank} is
\begin{equation}
\label{eq:cycle-rank}
\beta(H)=|E(H)|-|V(H)|+c(H).
\end{equation}
In particular, $H$ is a forest if and only if $\beta(H)=0$.

\section{Finiteness: obstructions and exact core reduction}
\label{sec:finiteness}

\subsection{The forest obstruction}

We first prove the necessary structural condition.

\begin{proof}[Proof of \cref{thm:necessary}]
Assume $\DN_{\LLY}(G)<\infty$ and choose a finite-distortion admissible positive weight $w$. Suppose that $\mathcal{S}(G)$ contains a cycle
\[
C=v_1v_2\cdots v_kv_1.
\]
Since each edge of $C$ belongs to $\Etl(G)$, one has $k\ge6$. Write
\[
a_i:=w_{v_iv_{i+1}},
\]
with indices modulo $k$.

For the edge $v_iv_{i+1}$, \cref{lem:tree-like-formula} gives
\[
a_i^2
\ge
\bigl(d_{v_i}^w-a_i\bigr)
\bigl(d_{v_{i+1}}^w-a_i\bigr).
\]
The first factor is at least $a_{i-1}$ and the second is at least $a_{i+1}$. Hence
\begin{equation}
\label{eq:cycle-logconvex}
a_i^2\ge a_{i-1}a_{i+1}
\qquad (i=1,\dots,k).
\end{equation}
Multiplying \eqref{eq:cycle-logconvex} over all $i$ gives equality of the two products. Since all terms are positive, equality must hold in every inequality in \eqref{eq:cycle-logconvex}. It must also hold in both preceding factor estimates. Therefore
\[
d_{v_i}^w-a_i=a_{i-1}
\qquad\text{for every }i.
\]
Thus no $v_i$ is incident with an edge outside $C$. The cycle is consequently a connected component of $G$, and connectedness gives $G=C_k$.

We have shown that if $\mathcal S(G)$ contains a cycle, then $G$
itself is a cycle of length at least $6$. Otherwise $\mathcal S(G)$
is a forest.
\end{proof}

\begin{example}[The forest condition is not sufficient]
\label{ex:forest-not-sufficient}
Let $G_6$ be obtained from the cycle
\[
xv_1v_2v_3v_4yx
\]
by adding a vertex $z$ adjacent to both $x$ and $y$. Then $xy,xz,yz$ lie in a triangle, whereas the five edges of the $x$--$y$ path through $v_1,\ldots,v_4$ are tree-like. Hence $\mathcal S(G_6)$ is that path together with the isolated vertex $z$, and is therefore a forest. Nevertheless,
\[
\DN_{\LLY}(G_6)=\infty.
\]

Indeed, suppose that a positive weight $w$ makes every edge nonnegatively curved. Write
\[
a_i=w_{v_{i-1}v_i}\quad(1\le i\le5),
\qquad v_0=x,\ v_5=y,
\]
and
\[
r=w_{xy},\qquad p=w_{xz},\qquad q=w_{yz}.
\]
Applying \eqref{eq:local-nonnegative} to the five tree-like path edges and multiplying the resulting inequalities gives
\begin{equation}
\label{eq:G6-path-product}
a_1a_5\ge(r+p)(r+q).
\end{equation}
Now define $f$ on $V(G_6)$ by
\[
f(x)=0,\quad f(v_1)=-1,\quad f(v_2)=0,\quad
f(v_3)=1,\quad f(v_4)=2,\quad f(y)=1,\quad f(z)=\frac12.
\]
The values differ by at most $1$ across every edge, so $f\in\operatorname{Lip}(1)$ and $f(y)-f(x)=1$. Since $\kappa_{\LLY}^w(x,y)\ge0$, \cref{eq:dual-LLY} gives
\[
0\le\Delta_wf(x)-\Delta_wf(y).
\]
After multiplying by the positive denominator $2(a_1+r+p)(a_5+r+q)$, this becomes
\[
0\le
4r^2+3r(p+q)+2pq-4a_1a_5-a_1q-a_5p.
\]
Using \eqref{eq:G6-path-product}, the right-hand side is at most
\[
-r(p+q)-2pq-a_1q-a_5p<0,
\]
a contradiction. Thus no admissible positive weight exists.
\end{example}

\begin{corollary}\label{cor:forest-consequences} Let $G$ be connected and locally finite with $\DN_{\LLY}(G)<\infty$, and suppose that \[ G\ne C_n\qquad\text{for every }n\ge6. \] Then every cycle of $G$ contains an edge that lies on a cycle of length $3$, $4$, or $5$. If $G$ is finite, then \begin{equation} \label{eq:forest-consequences} |\Etl(G)|\le |V(G)|-1, \qquad \beta(G)\le |E(G)\setminus\Etl(G)|. \end{equation} \end{corollary} \begin{proof} By \cref{thm:necessary}, $\mathcal S(G)$ is a forest. Hence every cycle of $G$ contains an edge outside $\Etl(G)$, and such an edge lies on a cycle of length $3$, $4$, or $5$. If $G$ is finite, the forest $\mathcal S(G)$ has at most $|V(G)|-1$ edges, giving the first inequality. Moreover, deleting $E(G)\setminus\Etl(G)$ leaves the forest $\mathcal S(G)$, so \[ |E(G)\setminus\Etl(G)|\ge\beta(G), \] which gives the second inequality. \end{proof}

\subsection{Exact \texorpdfstring{$2$}{2}-core reduction}

We first record the extension lemma needed to reconstruct the hanging trees from their boundary loads.

\begin{lemma}[Rooted tree extension]
\label{lem:rooted-tree-extension}
Let $G$ be a connected locally finite graph, and let $R\subseteq G$ be
a finite induced subtree rooted at $r$. Assume that
\[
N_G(v)\subseteq V(R)
\qquad\text{for every }v\in V(R)\setminus\{r\},
\]
so that $R$ meets the rest of $G$ only at its root $r$. Fix positive weights on all edges of $G$ incident with $r$ but not
belonging to $R$, and choose arbitrary positive weights on the edges of
$R$ incident with $r$. Then the remaining edges of $R$ can be assigned
positive weights so that every edge of $R$ has nonnegative
Lin--Lu--Yau curvature in $G$.
\end{lemma}

\begin{proof}
Orient $R$ away from $r$ and proceed outward. Suppose the parent edge
$uv$ of a vertex $v$ already has weight $a>0$, and put
\[
P_{uv}:=d_u^w-a\ge0.
\]
If $v$ has children, assign their edges positive weights with total
$S_v\le a^2/P_{uv}$ when $P_{uv}>0$, and any positive total when
$P_{uv}=0$; if $v$ is a leaf, set $S_v=0$. Since only the root meets
$G-R$,
\[
d_v^w-a=S_v,
\qquad
a^2\ge(d_u^w-a)(d_v^w-a).
\]
The bridge $uv$ is tree-like, so \cref{lem:tree-like-formula} gives
$\kappa_{\LLY}^w(uv)\ge0$. Proceeding recursively proves the claim.
\end{proof}

Let $G$ be finite and connected with $K=\core(G)\ne\varnothing$.
Every component of $G-V(K)$ is a tree attached to $K$ by one edge, so
every edge outside $K$ is a bridge. The next lemma aggregates all
first-layer branches at a core vertex into one pendant load.

\begin{lemma}
\label{lem:boundary-aggregation}
Let $G$ be finite and connected, let $K=\core(G)\ne\varnothing$, and let $w$ be a positive edge weight on $G$. For $x\in A_G$, put
\[
b_x:=\sum_{u\in N_G(x)\setminus V(K)}w_{xu}>0.
\]
Give $\widehat K_G$ the positive weight $\widehat w$ defined by
\[
\widehat w_e=w_e\quad(e\in E(K)),
\qquad
\widehat w_{xx^\partial}=b_x\quad(x\in A_G).
\]
Then for every $x\sim y\in E(K)$,
\[
\kappa_{\LLY}^{w,G}(x,y)
=
\kappa_{\LLY}^{\widehat w,\widehat K_G}(x,y).
\]
\end{lemma}

\begin{proof}
Put $U_z=N_G(z)\setminus V(K)$. Fix $x\sim y\in E(K)$. For any
$f\in\operatorname{Lip}(1)$ with $f(y)-f(x)=1$, keep the values of $f$
on $K$ and set
\[
\widehat f(z^\partial)
:=\frac1{b_z}\sum_{u\in U_z}w_{zu}f(u)
\qquad(z\in A_G).
\]
Since \[ f(u)\in[f(z)-1,f(z)+1] \qquad\text{for every }u\in U_z, \] and $\widehat f(z^\partial)$ is a weighted average of the values $f(u)$, it also lies in this interval. Hence \[ |\widehat f(z^\partial)-f(z)|\le1, \] so $\widehat f$ is $1$-Lipschitz on $\widehat K_G$. Since $b_z$ is the total weight of the edges from $z$ to $U_z$ and $\widehat f(z^\partial)$ is their weighted average, we have \[ d_z^{\widehat w}=d_z^w, \qquad \Delta_{\widehat w}\widehat f(z)=\Delta_w f(z) \quad(z=x,y). \] Hence, by \eqref{eq:dual-LLY}, \[ \kappa_{\LLY}^{w,G}(x,y)\ge \kappa_{\LLY}^{\widehat w,\widehat K_G}(x,y). \]

Conversely, let $\widehat f\in\operatorname{Lip}(1)$ on $\widehat K_G$ satisfy $\widehat f(y)-\widehat f(x)=1$, and extend it to $G$ by assigning the value $\widehat f(z^\partial)$ to every vertex in the component attached at $z$. The extension is $1$-Lipschitz, and each attached component contributes at its root exactly as the pendant edge $zz^\partial$. Hence the Laplacians at $x$ and $y$ are unchanged, and \eqref{eq:dual-LLY} gives the reverse inequality.
\end{proof}

\begin{proof}[Proof of \cref{thm:core-reduction}]
Let $K=\core(G)\ne\varnothing$. Suppose first that $w$ is admissible on $G$. Let $\widehat w$ be the weight on $\widehat K_G$ given by \cref{lem:boundary-aggregation}. Then every core edge has the same curvature under $w$ and $\widehat w$, while each added edge $xx^\partial$ is pendant and hence positively curved by \eqref{eq:leaf-curvature}. Thus $\widehat w$ is admissible on $\widehat K_G$. Conversely, let $\widehat w$ be admissible on $\widehat K_G$. Keep the weights on $E(K)$ and, for each $x\in A_G$, distribute the weight $\widehat w_{xx^\partial}$ among the first-layer edges of the hanging trees at $x$. Extend these weights through the hanging trees by \cref{lem:rooted-tree-extension}. Then all noncore edges have nonnegative curvature, while \cref{lem:boundary-aggregation} shows that the curvatures of the core edges agree with those under $\widehat w$. Hence the resulting weight on $G$ is admissible. 
\end{proof}

\subsection{A practical sufficient criterion}

We next give a simple sufficient condition for extending an admissible
weight from the $2$-core to the whole graph.

\begin{corollary}[Core extension criterion]
\label{cor:core-extension}
Let $G$ be finite and connected, and let
\[
K:=\core(G)\ne\varnothing.
\]
Suppose there exists a positive weight $w_0$ on $E(K)$ such that,
when $K$ is regarded as a weighted graph in its own right,
\[
\kappa_{\LLY}^{w_0,K}(e)\ge0
\qquad\text{for every }e\in E(K),
\]
and
\[
\kappa_{\LLY}^{w_0,K}(e)>0
\qquad\text{whenever }e\cap A_G\ne\varnothing.
\]
Then
\[
\DN_{\LLY}(G)<\infty.
\]
\end{corollary}

\begin{proof} If $A_G=\varnothing$, then $G=K$, and there is nothing to prove. For $\varepsilon>0$, let $w^\varepsilon$ be the weight on $\widehat K_G$ obtained by keeping $w_0$ on $E(K)$ and assigning weight $\varepsilon$ to every proxy edge $xx^\partial$, $x\in A_G$. We first claim that, for every $e\in E(K)$, \[ \kappa_{\LLY}^{w^\varepsilon,\widehat K_G}(e) \longrightarrow \kappa_{\LLY}^{w_0,K}(e) \qquad (\varepsilon\downarrow0). \] Fix $e=xy$. In the dual formula \eqref{eq:dual-LLY}, we may restrict to functions satisfying \[ f(x)=0,\qquad f(y)=1, \] since adding a constant changes neither the constraint nor the dual objective. Such $1$-Lipschitz functions are uniformly bounded on the finite graph $\widehat K_G$. For $z\in\{x,y\}\cap A_G$, \[ \Delta_{w^\varepsilon}f(z) = \frac{ \displaystyle \sum_{u\sim_K z}w_{0,zu}\bigl(f(u)-f(z)\bigr) + \varepsilon\bigl(f(z^\partial)-f(z)\bigr) }{ d_z^{w_0}+\varepsilon }. \] Since $|f(z^\partial)-f(z)|\le1$ and the functions under consideration are uniformly bounded, there is a constant $C>0$, independent of $f$ and sufficiently small $\varepsilon$, such that \[ \left| \Delta_{w^\varepsilon}f(z) - \Delta_{w_0}(f|_K)(z) \right| \le C\varepsilon. \] If $z\notin A_G$, the two Laplacians are identical. Hence, after enlarging $C$ if necessary, \[ \left| \bigl(\Delta_{w^\varepsilon}f(x)-\Delta_{w^\varepsilon}f(y)\bigr) - \bigl(\Delta_{w_0}(f|_K)(x)-\Delta_{w_0}(f|_K)(y)\bigr) \right| \le C\varepsilon. \] Since $f|_K$ is admissible in the dual formula on $K$, taking the infimum over $f$ gives \[ \kappa_{\LLY}^{w^\varepsilon,\widehat K_G}(e) \ge \kappa_{\LLY}^{w_0,K}(e)-C\varepsilon. \] Conversely, every $1$-Lipschitz function $g$ on $K$ with $g(y)-g(x)=1$ extends to $\widehat K_G$ by setting $g(z^\partial)=g(z)$ for $z\in A_G$. For this extension, \[ \Delta_{w^\varepsilon}g(z) = \Delta_{w_0}g(z)+O(\varepsilon) \qquad(z=x,y), \] uniformly in $g$. Taking infima in \eqref{eq:dual-LLY} therefore gives \[ \kappa_{\LLY}^{w^\varepsilon,\widehat K_G}(e) \le \kappa_{\LLY}^{w_0,K}(e)+O(\varepsilon). \] Thus \[ \kappa_{\LLY}^{w^\varepsilon,\widehat K_G}(e) \longrightarrow \kappa_{\LLY}^{w_0,K}(e). \] Now let \[ \eta := \min\left\{ \kappa_{\LLY}^{w_0,K}(e): e\in E(K),\ e\cap A_G\ne\varnothing \right\}>0. \]
For sufficiently small $\varepsilon>0$, every core edge meeting $A_G$
therefore has curvature at least $\eta/2$.
If $e\cap A_G=\varnothing$, its curvature is unchanged and hence remains nonnegative. Every proxy edge is pendant and therefore has positive curvature by \eqref{eq:leaf-curvature}. Thus $w^\varepsilon$ is admissible on $\widehat K_G$. By \cref{thm:core-reduction}, \[ \DN_{\LLY}(G)<\infty. \] \end{proof}

For a unicyclic graph the $2$-core is its unique cycle, so the core
extension criterion and the forest obstruction give a complete
classification.

\begin{corollary}[Unicyclic classification]
\label{cor:unicyclic-classification}
Let $G$ be a finite connected unicyclic graph, and let $C_n$ be its
unique cycle. Then
\[
\DN_{\LLY}(G)<\infty
\]
if and only if exactly one of the following holds:
\begin{enumerate}[label=\textup{(\roman*)}]
\item $3\le n\le5$;
\item $n\ge6$ and $G=C_n$.
\end{enumerate}
Moreover, in case \textup{(ii)},
\[
\DN_{\LLY}(G)=1.
\]
\end{corollary}

\begin{proof}
The $2$-core of $G$ is its unique cycle $C_n$.

Suppose first that $3\le n\le5$. Under the unit weight, a direct
calculation gives
\[
\kappa_{\LLY}(C_3)=\frac32,
\qquad
\kappa_{\LLY}(C_4)=1,
\qquad
\kappa_{\LLY}(C_5)=\frac12
\]
on every cycle edge. Thus the $2$-core has strictly positive
Lin--Lu--Yau curvature, and \cref{cor:core-extension} yields
\[
\DN_{\LLY}(G)<\infty.
\]

Now let $n\ge6$. Every edge of the unique cycle is tree-like in $G$.
Hence $\mathcal S(G)$ contains the cycle $C_n$. If
$\DN_{\LLY}(G)<\infty$, \cref{thm:necessary} therefore forces
\[
G=C_n.
\]
Conversely, if $G=C_n$, then the unit weight has zero curvature on
every edge by \eqref{eq:tree-like-curvature}. Hence
\[
\DN_{\LLY}(C_n)=1.
\]
\end{proof}

\section{Finite trees: fixed points and topology}
\label{sec:finite-trees}

\subsection{Finite-tree fixed-point formula}
\label{subsec:finite-fixed-point}

\begin{proof}[Proof of \cref{thm:tree-fixed-point}] By \cref{lem:rooted-tree-extension}, applied with $G=R=T$, there exists a positive admissible weight on $T$. Since $T$ is finite, rescaling gives a normalized admissible weight, hence a feasible vector in $\mathcal X_T$. Applying \cref{lem:monotone-iteration} therefore produces a fixed point of $F_T$. For uniqueness, let $x,y$ be fixed points and set \[ Rat:=\max_{e\in E(T)}\frac{x_e}{y_e}. \]
If $Rat>1$, choose $e=u\sim v$ attaining it. Then $e$ is non-pendant,
since every pendant edge has fixed-point coordinate $1$. Thus both
complementary loads at $e$ are at least $1$, so the maximum with $1$ in
\eqref{eq:distortion-map} is inactive at this coordinate for both fixed
points. Hence
\[
\mathrm{Rat}^2=\left(\frac{x_e}{y_e}\right)^2=
\frac{A_{u\mid v}(x)}{A_{u\mid v}(y)}
\frac{A_{v\mid u}(x)}{A_{v\mid u}(y)}\le
\frac{\mathrm{Rat}\,A_{u\mid v}(y)}{A_{u\mid v}(y)}
\frac{\mathrm{Rat}\,A_{v\mid u}(y)}{A_{v\mid u}(y)}
=\mathrm{Rat}^2,
\]
where the inequality follows from
$x_f\le \mathrm{Rat}\,y_f$ for every $f\in E(T)$.
Thus equality holds in both load estimates. Since all summands are
positive, every edge adjacent to $e$ has the same ratio $Rat$. Repeating
this along the finite tree reaches a pendant edge, whose two fixed-point
coordinates are both $1$, a contradiction. Hence $Rat\le1$; exchanging
$x$ and $y$ gives uniqueness.

Let the unique fixed point be $w_T^\star$. Iterating from any normalized
admissible weight $z$ gives $w_T^\star\le z$ coordinatewise. Since
$w_T^\star$ is itself admissible,
\[
\DN_{\LLY}(T)=\max_{e\in E(T)}w_T^\star(e).
\]

It remains to prove uniqueness of the normalized optimal weight. Put
$D=\DN_{\LLY}(T)$ and let $z$ be normalized and optimal. If $D=1$,
then $z\equiv1=w_T^\star$. Assume $D>1$ and choose an edge $e_0$ with
$w_T^\star(e_0)=D$. This edge is non-pendant. If $z_f>w_T^\star(f)$ for some edge $f$, let \[ f=e_1,e_2,\ldots,e_m=e_0 \] be the unique edge path from $f$ to $e_0$. We show inductively that \[ z_{e_i}>w_T^\star(e_i) \qquad (1\le i\le m). \] Suppose this holds for $e_i$, and let $e_{i+1}=u\sim v$. Since $e_i$ contributes to one of the complementary loads of $e_{i+1}$, say $A_{u\mid v}$, while $z\ge w_T^\star$ coordinatewise, \[ A_{u\mid v}(z)>A_{u\mid v}(w_T^\star), \qquad A_{v\mid u}(z)\ge A_{v\mid u}(w_T^\star). \] Hence feasibility of $z$ and the fixed-point equation for $w_T^\star$ give \[ z_{e_{i+1}}^2 \ge A_{u\mid v}(z)A_{v\mid u}(z) > A_{u\mid v}(w_T^\star)A_{v\mid u}(w_T^\star) = \bigl(w_T^\star(e_{i+1})\bigr)^2. \] Thus \[ z_{e_{i+1}}>w_T^\star(e_{i+1}), \] and in particular $z_{e_0}>D$, a contradiction. Therefore
$z=w_T^\star$. Scale invariance gives the equivalent unnormalized
statement.

Finally, the fixed-point equation gives zero curvature on non-pendant
edges, while pendant edges have weight $1$ and positive curvature.
\end{proof}

\subsection{Examples}
\label{subsec:tree-examples}

The fixed-point description becomes explicit for several small tree families. We record two examples in which the unweighted tree has a negatively curved non-pendant edge, so the distortion number is genuinely larger than $1$.

\begin{example}[Double stars]
\label{ex:double-star}
For integers $m,n\ge1$, let $ST_{m,n}$ denote the double star obtained from an edge $x\sim y$ by attaching $m$ leaves to $x$ and $n$ leaves to $y$. The edge $x\sim y$ is the unique non-pendant edge. Since the canonical weight $w^\star$ assigns weight $1$ to every pendant edge, if
\[
t:=w_{xy}^{\star},
\]
then the fixed-point equation is
\[
t^2=mn.
\]
Consequently,
\begin{equation}
\label{eq:double-star-DN}
\DN_{\LLY}(ST_{m,n})=\sqrt{mn}.
\end{equation}
When $m,n\ge2$, the unit weight is not admissible. Indeed, \[ \kappa_{\LLY}(x,y) =2\left(\frac{1}{m+1}+\frac{1}{n+1}\right)-2<0. \] Thus a nonuniform weighting is necessary.

This family also shows that the number of branch vertices does not
control the distortion number from above: $|B(ST_{m,n})|=2$ is fixed,
while \eqref{eq:double-star-DN} can be arbitrarily large.
\end{example}

\begin{example}[Three-center tree]
\label{ex:three-center}
Consider a tree $T$ with three branch vertices $x,y,z$ on a path $x\sim y\sim z$. Attach $m\ge2$ leaves to each of $x$ and $z$, and attach $r\ge1$ leaves to $y$. Let
\[
a=w_{xy}^{\star},
\qquad
b=w_{yz}^{\star},
\]
where $w^\star$ is the corresponding canonical weight.
The canonical fixed-point equations are
\begin{equation}
\label{eq:three-center-system}
a^2=m(r+b),
\qquad
b^2=m(r+a).
\end{equation}
Subtracting the two equations gives
\[
(a-b)(a+b+m)=0,
\]
so positivity forces $a=b$. Therefore the system reduces to
\[
a^2=m(r+a),
\]
and hence
\begin{equation}
\label{eq:three-center-DN}
\DN_{\LLY}(T)
=a
=\frac{m+\sqrt{m^2+4mr}}{2}.
\end{equation}
For $m=2$ and $r=1$, this gives
\[
\DN_{\LLY}(T)=1+\sqrt3.
\]
\end{example}

\subsection{suppression}
\label{subsec:orientation-suppression}

\begin{lemma}[One-vertex suppression]
\label{lem:one-suppression}
Let $T'$ be obtained from a finite tree $T$ by suppressing one degree-two vertex. Then
\[
\DN_{\LLY}(T')\le\DN_{\LLY}(T).
\]
\end{lemma}

\begin{proof}
Let $v$ have neighbors $x,y$ and put $D=\DN_{\LLY}(T)$. Assume first
that neither $x$ nor $y$ is a leaf. For the canonical weight set
\[
a=w_{xv}^\star,
\qquad b=w_{vy}^\star,
\qquad
A_x=d_x^{w^\star}-a,
\qquad
A_y=d_y^{w^\star}-b.
\]
Zero curvature on $xv$ and $vy$ gives
\[
a^2=A_xb,
\qquad b^2=aA_y,
\qquad A_xA_y=ab.
\]
By symmetry assume $a\le b$. After deleting $v$, multiply all old
weights in the $x$-component by $b/a$, leave the $y$-component
unchanged, and give the new edge $xy$ weight $b$. Old curvatures are
unchanged, while
\[
\left(\frac ba A_x\right)A_y=b^2,
\]
so $xy$ has zero curvature.

It remains to control distortion when $a<b$. From $b^2>a^2=A_xb$ we get
$d_x^{w^\star}<d_v^{w^\star}$, so $xv$ is oriented from $x$ to $v$.
By \cref{lem:one-outgoing}, orientation propagates through the
$x$-component toward $v$; hence every non-pendant edge in the $x$-component has
canonical weight at most $a$. After multiplication by $b/a$, all its
weights lie in $[1,b]$. The other component and the new edge have
weights in $[1,D]$, so the new distortion is at most $D$. The case
$a=b$ is immediate.

If, say, $x$ is a leaf, then $w_{xv}^\star=1$. Delete $v$, give the new
pendant edge $xy$ the old weight $w_{vy}^\star$, and leave all other
weights unchanged. The new edge is positively curved and the distortion
is still at most $D$. The remaining cases are symmetric. Thus
$\DN_{\LLY}(T')\le D$.
\end{proof}

\begin{proof}[Proof of \cref{thm:suppression-intro}]
Apply \cref{lem:one-suppression} repeatedly until no degree-two vertex remains.
\end{proof}

\begin{example}[Strict suppression]
\label{ex:strict-suppression}
Let $T=ST_{2,1}$. By \cref{ex:double-star}, \[ \DN_{\LLY}(T)=\sqrt2. \] Suppressing the degree-two center of the $1$-leaf side produces the star $K_{1,3}$, for which \[ \DN_{\LLY}(K_{1,3})=1. \] Hence the inequality in \cref{thm:suppression-intro} can be strict.
\end{example}

\subsection{Branch vertices}
\label{subsec:branch-vertices}

We now turn to branch vertices. 
Set
\[
L(T):=\{v\in V(T):d_v=1\},
\qquad
\ell_T(v):=|N_T(v)\cap L(T)|.
\]
Let $S$ be a tree with no degree-two vertices and at least two branch vertices, and let $w^\star$ be the canonical weight of $S$. Its non-pendant edges form a tree $H$ on $B(S)$; we call $H$ the \emph{branch skeleton} of $S$. For $v\in B(S)$, set
\[
s_v:=d_v^{w^\star}.
\]
For a branch edge $u\sim v\in E(H)$, by \eqref{eq:Aue}, we have
\[
A_{u\mid v}=\ell_S(u)+\sum_{\substack{u\sim z\in E(H)\\z\ne v}}w_{uz}^\star.
\]
Set
\[
p(u,v):=\frac{w_{uv}^\star}{w_{uv}^\star+A_{u\mid v}}.
\]
The zero-curvature equation \eqref{eq:tree-like-curvature} gives
\begin{equation}
\label{eq:p-sum-one}
p(u,v)+p(v,u)=1.
\end{equation}
Thus, by \cref{lem:one-outgoing}, $u\sim v$ is oriented from $u$ toward $v$ exactly when $p(u,v)>1/2$. Ties are left unoriented.

For $e=u\sim v\in E(H)$, let $H_{u\mid v}$ denote the component of $H-e$ containing $u$, and put
\[
m(u\mid v):=|V(H_{u\mid v})|.
\]
We next estimate the complementary load $A_{u\mid v}$ in terms of the number of branch vertices on the $u$-side of $uv$.

\begin{lemma}
\label{lem:strong-rooted-load}
Let $u\sim v\in E(H)$ satisfy $p(u,v)\ge1/2$. Then
\begin{equation}
\label{eq:strong-rooted-load}
A_{u\mid v}\ge 2m(u\mid v).
\end{equation}
Moreover,
\[
w_{uv}^\star>A_{u\mid v}\ge2m(u\mid v)
\]
when $p(u,v)>1/2$, while
\[
w_{uv}^\star=A_{u\mid v}\ge2m(u\mid v)
\]
when $p(u,v)=1/2$.
\end{lemma}

\begin{proof}
Write $W=w_{uv}^\star$, $m=m(u\mid v)$, and
$\ell=\ell_S(u)$. Since $p(u,v)\ge1/2$,
\begin{equation}
\label{eq:boundary-dominates-load}
W\ge A_{u\mid v},
\end{equation}
with strict inequality exactly when $p(u,v)>1/2$. We prove
$A_{u\mid v}\ge2m$ by strong induction on $m$.

If $m=1$, the vertex $u$ has no branch neighbor on its side of the
edge, so $\ell\ge2$ and $A_{u\mid v}=\ell\ge2$.

Let $m>1$, and assume that the assertion holds whenever $m(u'\mid v')<m$ for an edge $u'\sim v'\in E(H)$ satisfying $p(u',v')\ge1/2$. Let $u_1,\ldots,u_c$ be the branch neighbors of $u$ in
$H_{u\mid v}$. Put
\[
w_i=w_{uu_i}^\star,
\qquad
A_i=A_{u_i\mid u},
\qquad
m_i=m(u_i\mid u).
\]
Each edge $u_i u$ is strictly oriented toward $u$. Indeed, if $uv$ is
strict, this follows from \cref{lem:one-outgoing}; if $uv$ is a tie,
\cref{lem:one-outgoing} excludes an outgoing orientation of $u_i u$,
while $u_i u$ cannot also be a tie since $u$ is a branch vertex. Hence the induction hypothesis gives
\begin{equation}
\label{eq:child-strong-load}
A_i\ge2m_i,
\qquad
w_i>A_i\ge2m_i.
\end{equation}
In particular, $w_i>2$.

Set
\[
R_i:=\ell+\sum_{j\ne i}w_j,
\qquad
B_i:=A_{u\mid u_i}=W+R_i.
\]
Since $W\ge w_i+R_i$,
\begin{equation}
\label{eq:Bi-gap}
q_i:=B_i-w_i\ge2R_i.
\end{equation}
From $w_i^2=A_iB_i$,
\begin{equation}
\label{eq:delta-formula}
\delta_i:=w_i-A_i
=\frac{w_iq_i}{w_i+q_i}.
\end{equation}

If $\ell\ge2$, then \eqref{eq:child-strong-load} immediately gives
\[
A_{u\mid v}=\ell+\sum_iw_i
\ge2+2\sum_i m_i=2m.
\]
If $\ell=1$, then $c\ge1$ and \eqref{eq:Bi-gap} gives $q_i\ge2$; since
$w_i>2$, \eqref{eq:delta-formula} gives $\delta_i>1$. Therefore
\[
A_{u\mid v}
=1+\sum_i(A_i+\delta_i)
>1+2\sum_i m_i+1=2m.
\]
If $\ell=0$, then $c\ge2$, so $R_i>2$, $q_i>4$, and again
$\delta_i>1$. Hence
\[
A_{u\mid v}
=\sum_i(A_i+\delta_i)
>2\sum_i m_i+c\ge2m.
\]
This proves the induction. The final strict/equality statements follow
from the relation between $W$ and $A_{u\mid v}$ above.
\end{proof}

In particular, if $D=\DN_{\LLY}(S)$ and a branch edge $u\sim v$ is strictly oriented from $u$ toward $v$, then
\[
2m(u\mid v)<w_{uv}^\star\le D,
\qquad\text{and hence}\qquad
m(u\mid v)<\frac{D}{2}.
\]

We next describe the terminal configuration of the orientation. Recall that an edge is a tie precisely when its endpoints have the same canonical weighted degree.

\begin{lemma}[Terminal configuration]
\label{lem:terminal-configuration}
Let $H$ be the branch skeleton of $S$. Exactly one of the following occurs:
\begin{enumerate}[label=\textup{(\roman*)}]
\item there is a unique vertex $r\in V(H)$ with maximal weighted degree $s_r$, and every edge incident with $r$ is strictly oriented toward $r$;
\item there are exactly two vertices $x,y\in V(H)$ with maximal weighted degree, they are adjacent, and $x\sim y$ is a tie edge.
\end{enumerate}
\end{lemma}

\begin{proof}
A branch vertex cannot be incident with two tie edges, since each such edge would have weight equal to half the weighted degree of the branch vertex, leaving no weight for any other incident edge. Let $M$ be the set of vertices of $H$ with maximal weighted
degree. If $x,y\in M$ are distinct, consider their unique path $x=v_0,\ldots,v_k=y$.
If some interior vertex has weighted degree smaller than the maximum,
let
\[
s_*:=\min_{0\le t\le k}d_{v_t}^{w^*},
\]
and choose a maximal consecutive block:
\[v_i,\ldots,v_j\] such that
\[
d_{v_i}^{w^*}=\cdots=d_{v_j}^{w^*}=s_*,
\]
with $d_{v_{i-1}}^{w^*}>s_*$ when $i>0$ and
$d_{v_{j+1}}^{w^*}>s_*$ when $j<k$. A one-vertex block gives two strict outgoing edges; a
longer block gives a tie edge together with a strict outgoing edge.
Both contradict \cref{lem:one-outgoing}. Thus the weighted degree is
constant along the path. Since two consecutive tie edges cannot meet at
a branch vertex, $k=1$. Hence $|M|\le2$, and two maximal vertices, if
present, are adjacent and joined by a tie. The stated alternatives
follow.
\end{proof}

The final aggregation uses the complementary quantity
\[
\theta_{u\mid v}
:=\frac{A_{u\mid v}}{A_{u\mid v}+w_{uv}^\star}.
\]
By \eqref{eq:p-sum-one}, $\theta_{u\mid v}=p(v,u)$. In particular, in case \textup{(i)} of \cref{lem:terminal-configuration}, if $u$ is adjacent to the unique maximal vertex $r$, then
\begin{equation}
\label{eq:theta-root}
\theta_{u\mid r}
=p(r,u)
=\frac{w_{ur}^\star}{s_r}.
\end{equation}

\begin{proposition}[Branch bound without degree-two vertices]
\label{prop:branch-reduced}
Let $S$ be a finite tree without degree-two vertices. Then
\[
|B(S)|\le\left\lceil\DN_{\LLY}(S)\right\rceil.
\]
\end{proposition}

\begin{proof}
Put $b=|B(S)|$ and $D=\DN_{\LLY}(S)$; assume $b\ge2$. In case
\textup{(i)} of \cref{lem:terminal-configuration}, let $u_1,\ldots,u_k$ be the branch neighbors of the unique vertex
$r$ of maximal weighted degree, and set
\[
m_i=m(u_i\mid r),
\qquad
A_i=A_{u_i\mid r},
\qquad
w_i=w_{u_i r}^\star,
\qquad
\theta_i=\frac{A_i}{A_i+w_i}.
\]
By \cref{lem:strong-rooted-load}, $A_i\ge2m_i$ and $w_i>A_i$. Since
$D\ge w_i$,
\[
D\theta_i
\ge \frac{w_iA_i}{A_i+w_i}
>\frac{A_i}{2}\ge m_i.
\]
The rooted components partition $V(H)\setminus\{r\}$, while
\eqref{eq:theta-root} gives
\[
b-1=\sum_i m_i,
\qquad
\sum_i\theta_i
=\frac{\sum_iw_i}{s_r}\le1.
\]
Hence $b-1<D$.

In case \textup{(ii)}, let $xy$ be the terminal tie and put
$m_x=m(x\mid y)$, $m_y=m(y\mid x)$, and $w=w_{xy}^\star$. Then
\[
w=A_{x\mid y}=A_{y\mid x}
\]
and \cref{lem:strong-rooted-load} gives $w\ge2m_x,2m_y$. Therefore
\[
b=m_x+m_y\le2\max\{m_x,m_y\}\le w\le D.
\]
In both cases $b<D+1$, hence $b\le\lceil D\rceil$.
\end{proof}

\begin{proof}[Proof of \cref{thm:branch-intro}]
Let $S=\overline T$. By \eqref{eq:branch-preserved},
\cref{prop:branch-reduced}, and \cref{thm:suppression-intro},
\[
|B(T)|=|B(S)|
\le\left\lceil\DN_{\LLY}(S)\right\rceil
\le\left\lceil\DN_{\LLY}(T)\right\rceil.
\]
\end{proof}

\begin{remark}
The bound is sharp: \cref{ex:double-star} gives equality at $2$, and
\cref{ex:three-center} with $m=2,r=1$ gives equality after taking the ceiling. For a tree $S$ without degree-two vertices, the proof
also yields
\[
|B(S)|-1<\DN_{\LLY}(S)
\]
when the terminal configuration has a unique vertex of maximal weighted degree, and
$|B(S)|\le\DN_{\LLY}(S)$ when the terminal configuration has a tie edge.
\end{remark}

\section{Infinite trees: fixed points, exhaustion, and ends}
\label{sec:infinite-trees}

\subsection{Bounded fixed points and exhaustion}
\label{subsec:infinite-fixed-exhaustion}

For a locally finite infinite tree $T$, write $\mathcal F_b(T)$ for
the set of bounded normalized fixed points of $F_T$ and put
\[
\|x\|_\infty:=\sup_{e\in E(T)}x_e.
\]
All terminology is as in \cref{def:distortion-map}.

We first record the monotonicity that makes exhaustion possible.

\begin{lemma}[Restriction to a subtree]
\label{lem:infinite-restriction}
Let $T$ be a locally finite infinite tree, let $S\subseteq T$ be a subtree, and let $w$ be a positive edge weight on $T$. Then for every $uv\in E(S)$,
\begin{equation}
\label{eq:restriction-curvature}
\kappa_{\LLY}^{w|_S,S}(u,v)
\ge
\kappa_{\LLY}^{w,T}(u,v).
\end{equation}
Consequently, the restriction of an admissible weight to $S$ is admissible and
\[
\Dist(w|_S)\le\Dist(w)
\]
whenever the right-hand side is finite.
\end{lemma}

\begin{proof}
By \eqref{eq:tree-like-curvature}, restriction leaves $w_{uv}$ unchanged
and can only decrease the weighted degrees at $u$ and $v$, so the
curvature can only increase. The other assertions are immediate.
\end{proof}

\begin{lemma}[Bounded fixed-point characterization]
\label{lem:infinite-fixed-criterion}
Let $T$ be a locally finite infinite tree. Then
\[
\DN_{\LLY}(T)<\infty
\quad\Longleftrightarrow\quad
\mathcal F_b(T)\ne\varnothing.
\]
When these equivalent conditions hold,
\begin{equation}
\label{eq:infinite-fixed-inf}
\DN_{\LLY}(T)
=
\inf_{x\in\mathcal F_b(T)}\|x\|_\infty.
\end{equation}
\end{lemma}

\begin{proof}
Every bounded normalized fixed point is admissible, and hence
\[
\DN_{\LLY}(T)\le \|x\|_\infty
\qquad (x\in\mathcal F_b(T)).
\]
Conversely, if $y$ is any bounded normalized admissible weight, then
$y$ is a feasible vector of \cref{def:distortion-map} by
\eqref{eq:supersolution}, and \cref{lem:monotone-iteration} produces a bounded
normalized fixed point $x^\star\le y$. Therefore
\[
\inf_{x\in\mathcal F_b(T)}\|x\|_\infty
\le \|x^\star\|_\infty
\le \|y\|_\infty
=\Dist(y).
\]
Taking the infimum over all such $y$ gives
\[
\DN_{\LLY}(T)
=
\inf_{x\in\mathcal F_b(T)}\|x\|_\infty,
\]
and also proves the stated equivalence.
\end{proof}

\begin{lemma}[Exhaustion] \label{lem:infinite-exhaustion} Let $T$ be a locally finite infinite tree, and let \[ T_1\subset T_2\subset\cdots, \qquad \bigcup_{n\ge1}T_n=T, \] be an exhaustion by finite subtrees. Then \[ \DN_{\LLY}(T_n)\nearrow\DN_{\LLY}(T). \] If $\DN_{\LLY}(T)<\infty$, then a subsequence of the canonical weights $w_{T_n}^\star$ converges pointwise to a bounded normalized optimal fixed point of $F_T$. \end{lemma}

\begin{proof}
Set $D_n=\DN_{\LLY}(T_n)$. By \cref{lem:infinite-restriction},
\[
D_n\le D_{n+1}\le\DN_{\LLY}(T),
\]
so $D_n\uparrow D_\infty\le\DN_{\LLY}(T)$. If $D_\infty=\infty$, the
claim follows. Assume $D_\infty<\infty$ and let $w^{(n)}$ be the
canonical weight of $T_n$. Then
\begin{equation}
\label{eq:finite-exhaustion-bounds}
1\le w_e^{(n)}\le D_n\le D_\infty.
\end{equation}
Enumerate \[ E(T)=\{e_1,e_2,\ldots\}. \] By \eqref{eq:finite-exhaustion-bounds}, for each fixed $e_j$ the sequence $w_{e_j}^{(n)}$, defined for all sufficiently large $n$, lies in the compact interval $[1,D_\infty]$. By successively passing to nested subsequences for $e_1,e_2,\ldots$ and then taking the diagonal subsequence, we obtain indices $n_k\to\infty$ such that $w_e^{(n_k)}$ converges for every $e\in E(T)$. Define \[ w_e^{(\infty)} := \lim_{k\to\infty}w_e^{(n_k)}. \]

Fix $e=u\sim v$. By local finiteness, all edges incident with $u$ or $v$ belong to $T_{n_k}$ for all sufficiently large $k$. Hence the complementary loads in the fixed-point equation are eventually sums over fixed finite sets of edges. Since $w_f^{(n_k)}\to w_f^{(\infty)}$ for every $f\in E(T)$, we have \[ A_{u\mid v}(w^{(n_k)}) \longrightarrow A_{u\mid v}(w^{(\infty)}), \qquad A_{v\mid u}(w^{(n_k)}) \longrightarrow A_{v\mid u}(w^{(\infty)}). \]
Passing to the limit in \eqref{eq:distortion-map} therefore gives \[ w_e^{(\infty)} = \max\left\{1, \sqrt{A_{u\mid v}(w^{(\infty)}) A_{v\mid u}(w^{(\infty)})} \right\}. \]
Thus $w^{(\infty)}$ is a bounded fixed point, and
\eqref{eq:finite-exhaustion-bounds} gives
\[
\DN_{\LLY}(T)
\le\Dist(w^{(\infty)})
\le D_\infty.
\]
Hence equality holds throughout and
$D_n\uparrow\DN_{\LLY}(T)$. Moreover
$\Dist(w^{(\infty)})=D_\infty$ while all its coordinates lie in
$[1,D_\infty]$, so $\inf_e w_e^{(\infty)}=1$. Thus
$w^{(\infty)}$ is a bounded normalized optimal fixed point. 
\end{proof}

\subsection{Ends and classification}
\label{subsec:infinite-ends}

A \emph{ray} is a simple path
\[
v_0v_1v_2\cdots,
\]
while a \emph{double ray} is a simple path
\[
\cdots v_{-1}v_0v_1\cdots
\]
indexed by $\mathbb Z$. Let $R_1,R_2\subseteq G$ be two rays. We say $R_1$ is a \emph{degree-two ray} of $G$ if $d_G(v)=2$ for any vertex $v\in V(R_1)$. They determine the same
\emph{end} if, for every finite set $S\subseteq V(G)$, sufficiently
long tails of $R_1$ and $R_2$ lie in the same connected component of
$G-S$. For a tree, this is equivalent to saying that the two rays eventually coincide; that is, they have a common tail.

The finite branch-vertex theorem \cref{thm:branch-intro} immediately yields a global restriction in the infinite setting.

\begin{lemma}[Infinite branch bound]
\label{lem:infinite-branch-set}
Let $T$ be a locally finite infinite tree and suppose
\[
D:=\DN_{\LLY}(T)<\infty.
\]
Then $B(T)$ is finite and
\begin{equation}
\label{eq:infinite-branch-bound}
|B(T)|\le\lceil D\rceil.
\end{equation}
\end{lemma}

\begin{proof}
Let $T_n\nearrow T$ be a finite connected exhaustion. By
\cref{lem:infinite-exhaustion},
$\DN_{\LLY}(T_n)\le D$. For any finite $F\subset B(T)$, local
finiteness implies $F\subset B(T_n)$ for all large $n$. Hence
\[
|F|\le |B(T_n)|
\le\left\lceil\DN_{\LLY}(T_n)\right\rceil
\le\lceil D\rceil
\]
by \cref{thm:branch-intro}. Applying this to every finite
$F\subset B(T)$ proves the claim.
\end{proof}

We next describe the behavior of a bounded zero-curvature weight along a degree-two ray.

\begin{lemma}[Rigidity along a degree-two ray]
\label{lem:ray-rigidity}
Let $T$ be a locally finite tree and let $w$ be a positive edge weight satisfying
\[
0<\inf_e w_e\le\sup_e w_e<\infty
\]
and
\[
\kappa_{\LLY}^w(e)=0
\qquad\text{for every non-pendant edge }e.
\]
Suppose
\[
v_0v_1v_2\cdots
\]
is a ray such that $d_{v_j}=2$ for every $j\ge1$. Put
\[
a_j:=w_{v_jv_{j+1}}.
\]
Then
\begin{equation}
\label{eq:ray-constant}
a_0=a_1=a_2=\cdots.
\end{equation}
If $v_0$ is a branch vertex, then the edge $v_0v_1$ is a tie in weighted degree:
\begin{equation}
\label{eq:end-tie}
d_{v_0}^w=d_{v_1}^w=2a_0.
\end{equation}
\end{lemma}

\begin{proof}
Zero curvature gives \[ a_j^2=a_{j-1}a_{j+1} \qquad (j\ge1), \] and hence \[ \frac{a_{j+1}}{a_j} = \frac{a_j}{a_{j-1}}. \] Thus the ratio $a_j/a_{j-1}$ is independent of $j$. Writing $\rho:={a_1}/{a_0}$, we obtain \[ a_j=a_0\rho^j \qquad (j\ge0). \] If $\rho>1$, then $a_j\to\infty$, while if $\rho<1$, then $a_j\to0$. Both contradict \[ 0<\inf_e w_e\le\sup_e w_e<\infty. \] Therefore $\rho=1$, and hence $a_j=a_0$ for all $j\ge0$.

If $v_0$ is a branch vertex, put
$A=\sum_{z\sim v_0,\,z\ne v_1}w_{v_0z}$. The zero-curvature equation on
$v_0v_1$ becomes $a_0^2=Aa_1=Aa_0$, so $A=a_0$. Therefore
$d_{v_0}^w=d_{v_1}^w=2a_0$.
\end{proof}

\begin{proof}[Proof of \cref{thm:ends-intro}]

Assume first that $D:=\DN_{\LLY}(T)<\infty$. By
\cref{lem:infinite-exhaustion}, $T$ has a bounded optimal
fixed-point weight $w$, so every non-pendant edge has zero curvature.
By \cref{lem:infinite-branch-set}, $B(T)$ is finite and satisfies the
stated bound.

If $B(T)=\varnothing$, then $T$ is a ray or a double ray; the ray is
case \textup{(ii)} with a one-vertex finite initial tree. Assume now
$B(T)\ne\varnothing$. We show that $T$ has one end. If $T$ had two distinct ends, choose rays representing them and let $r_1,r_2$ be their last branch vertices. Then every later vertex on each ray has degree $2$, so \cref{lem:ray-rigidity} gives a tie edge leaving each $r_i$ toward its end. A
branch vertex cannot support two ties, since each tie has weight half
its weighted degree; hence $r_1\ne r_2$.

Let $r_1=u_0,\ldots,u_k=r_2$ be the unique path between them and set
$s_j=d_{u_j}^w$. The endward tie at $r_1$ uses weight $s_0/2$, so every
other incident edge has smaller weight; by \eqref{eq:harmonic-edge},
$s_1<s_0$. Similarly $s_{k-1}<s_k$. Analogous to the techniques employed in \cref{lem:terminal-configuration}, choose a maximal block on which
$s_0,\ldots,s_k$ attains its minimum. A one-vertex block gives two
strict outgoing edges, while a longer block gives a tie together with
a strict outgoing edge, both contradicting \cref{lem:one-outgoing}.
Thus $T$ cannot have two ends.

Since $B(T)$ is finite, choose a ray $R=v_0v_1\cdots$ representing the
unique end with $v_0$ beyond all branch vertices. Then $R$ is a degree-two ray of $T$. The component $K$
containing $v_0$ after deleting $v_0v_1$ is finite. Indeed, if $K$ were
infinite, then starting at $v_0$ one could recursively choose a neighbor
whose remaining component is infinite; local finiteness would thereby
produce a ray in $K$, and hence a second end of $T$. Thus $T$ is obtained
from the finite tree $K$ by attaching $R$.

Conversely, the double ray has unit weight and zero curvature. Suppose
$T$ is obtained from a finite tree $K$ by attaching a degree-two ray
$R=v_0v_1\cdots$. Note that $v_0\in V(K)$. Let $v\in V(K)$ with $v\sim v_0$. Give the ray edges weight $1$ and choose positive
weight on the edge $vv_0$ with $w_{vv_0}<1$. Apply
\cref{lem:rooted-tree-extension} to $K$, treating $v_0v_1$ as fixed
ambient data. The edges of $K$ are then nonnegatively curved. Note that the first
ray edge has complementary loads $<1$ and $1$. By \eqref{eq:local-nonnegative}, $v_0v_1$ are positively curved and every later ray edge
has zero curvature. Only finitely many weights differ from $1$, so the
distortion is finite. The branch and end assertions follow from the two
alternatives.
\end{proof}

We next prove \cref{thm:infinite-fixed-exhaustion-intro}.

\begin{proof}[Proof of \cref{thm:infinite-fixed-exhaustion-intro}] The equivalence between finite distortion and the existence of a bounded normalized fixed point is \cref{lem:infinite-fixed-criterion}. Assume henceforth that \[ D:=\DN_{\LLY}(T)<\infty. \] By \cref{lem:infinite-exhaustion}, $F_T$ has a bounded normalized optimal fixed point. We first prove uniqueness of the bounded normalized fixed point. If $T$ is a double ray, let $x$ be a bounded normalized fixed point of $F_T$, and write \[ x_i:=x_{v_iv_{i+1}} \qquad (i\in\mathbb Z) \] for its edge weights. The fixed-point equations give \[ x_i^2=x_{i-1}x_{i+1}, \] and hence \[ \frac{x_{i+1}}{x_i} = \frac{x_i}{x_{i-1}}. \] Thus the ratio of consecutive weights is constant. Boundedness in both directions forces this ratio to be $1$, and normalization then gives $x_i=1$ for every $i$. Suppose now that $T$ is in case \textup{(ii)} of \cref{thm:ends-intro}. By \cref{lem:ray-rigidity}, every bounded fixed point is constant along the degree-two ray outside a finite subtree. Hence, for two bounded normalized fixed points $x,y$, the quantity \[ Rat:=\sup_{e\in E(T)}\frac{x_e}{y_e} \] is attained. If $Rat>1$, choose an edge $e=u\sim v$ attaining it. The edge $e$ is non-pendant, since every fixed point has weight $1$ on pendant edges. Since $x_f\le Raty_f$ for every $f\in E(T)$, the fixed-point equations give \[ Rat^2 = \frac{A_{u\mid v}(x)}{A_{u\mid v}(y)} \frac{A_{v\mid u}(x)}{A_{v\mid u}(y)} \le Rat^2. \] Thus equality holds in both load estimates, and every edge adjacent to $e$ has the same ratio $R$. Repeating this along the tree reaches a pendant edge, where both fixed-point weights are $1$, a contradiction. Hence $x\le y$ coordinatewise. Exchanging $x$ and $y$ gives $x=y$. 

Denote the unique bounded normalized fixed point by $w_T^\star$. The optimal fixed point supplied by \cref{lem:infinite-exhaustion} must therefore be $w_T^\star$. Hence
\[ \DN_{\LLY}(T) = \Dist(w_T^\star) = \sup_{e\in E(T)}w_T^\star(e).
\]
We next prove uniqueness of the normalized optimal weight. Let $w$ be any normalized optimal weight. Since $w$ is bounded and feasible when regarded as a vector in $\mathcal X_T$, \cref{lem:monotone-iteration} produces a bounded normalized fixed point $x^\star\le w$. By uniqueness of the bounded normalized fixed point, \[ x^\star=w_T^\star, \] and hence \[ w\ge w_T^\star \] coordinatewise. If $D=1$, then $w\equiv1=w_T^\star$. Assume $D>1$. By \cref{thm:ends-intro,lem:ray-rigidity}, $w_T^\star$ is eventually constant outside a finite subtree, so its supremum is attained. Choose $e_0$ with \[ w_T^\star(e_0)=D. \] Suppose that $w_f>w_T^\star(f)$ for some edge $f$. Analogous to the techniques employed in \cref{thm:tree-fixed-point}, this strict inequality propagates along the unique edge path from $f$ to $e_0$. Thus \[ w_{e_0}>D, \] contradicting the normalization of $w$ and $\Dist(w)=D$. Therefore $w=w_T^\star$. Scale invariance gives uniqueness up to a global positive scalar among all optimal weights. 

Finally, let $(T_n)$ be any exhaustion of $T$. By \cref{lem:infinite-exhaustion}, \[ \DN_{\LLY}(T_n)\nearrow D. \] Every subsequence of the canonical weights $w_{T_n}^\star$ has, by the same diagonal argument used in \cref{lem:infinite-exhaustion}, a pointwise convergent subsubsequence. Let $\overline w$ be its limit. As in the proof of \cref{lem:infinite-exhaustion}, $\overline w$ is a bounded fixed point and \[ 1\le \overline w_e\le D \qquad(e\in E(T)). \] Therefore \[ D = \DN_{\LLY}(T) \le \Dist(\overline w) \le D. \] Hence $\overline w$ is a bounded normalized optimal fixed point, so by uniqueness \[ \overline w=w_T^\star. \] 
Thus every subsequence has a subsubsequence converging pointwise to $w_T^\star$, and consequently \[ w_{T_n}^\star(e)\longrightarrow w_T^\star(e) \qquad(e\in E(T)). \] \end{proof}

\begin{remark}[Nonuniqueness on general graphs]
This uniqueness phenomenon is special to trees. Let $G$ be the graph
with
\[
V(G)=\{x,y,a,b,c,d\}
\]
and
\[
E(G)=\{xy,xa,xb,ab,yc,yd\}.
\]
For $s\in[1,2]$, define
\[
w^{(s)}_{xy}=2,
\qquad
w^{(s)}_{xa}=w^{(s)}_{xb}=w^{(s)}_{yc}=w^{(s)}_{yd}=1,
\qquad
w^{(s)}_{ab}=s.
\]
The bridge $xy$ is tree-like and the edges $yc,yd$ are pendant, so
\cref{lem:tree-like-formula} and \eqref{eq:leaf-curvature} give
\[
\kappa_{\LLY}^{w^{(s)}}(x,y)=0,
\qquad
\kappa_{\LLY}^{w^{(s)}}(y,c)
=
\kappa_{\LLY}^{w^{(s)}}(y,d)
=\frac12.
\]
For the edge $xa$, normalize a dual-feasible function by $f(x)=0$ and
$f(a)=1$. The edge constraints force $f(b)\in[0,1]$, and the objective
is minimized by $f(y)=-1$ and $f(b)=1$; the values on the two leaves at
$y$ can then be chosen to give a feasible $1$-Lipschitz extension. This
gives
\[
\kappa_{\LLY}^{w^{(s)}}(x,a)
=
\kappa_{\LLY}^{w^{(s)}}(x,b)
=\frac{1}{s+1}.
\]
For $ab$, after normalizing $f(a)=0$ and $f(b)=1$, the dual objective is
independent of the remaining feasible values and equals
\[
\kappa_{\LLY}^{w^{(s)}}(a,b)
=2-\frac{1}{s+1}.
\]
Hence every $w^{(s)}$ is admissible and has distortion $2$.

On the other hand, $xy$ is a bridge and hence tree-like, with
$d_x=d_y=3$. Let $\widetilde w$ be any finite-distortion admissible
weight and put
\[
m:=\inf_{e\in E(G)}\widetilde w_e,
\qquad
M:=\sup_{e\in E(G)}\widetilde w_e.
\]
Applying \eqref{eq:local-nonnegative} to $xy$ gives
\[
\widetilde w_{xy}^2
\ge
(d_x^{\widetilde w}-\widetilde w_{xy})
(d_y^{\widetilde w}-\widetilde w_{xy})
\ge4m^2.
\]
Since $\widetilde w_{xy}\le M$, it follows that
\[
\Dist(\widetilde w)=\frac{M}{m}\ge2.
\]
Therefore $\DN_{\LLY}(G)\ge2$. Combined with the admissible family
above, this yields
\[
\DN_{\LLY}(G)=2,
\]
and the normalized optimal weights contain the nontrivial continuum
$\{w^{(s)}:1\le s\le2\}$.
\end{remark}

\section{Tree-like skeletons in connected graphs}
\label{sec:tree-like-skeleton}

\subsection{Comparison and lower bounds}
\label{subsec:skeleton-comparison}

We now transfer the tree theory to the tree-like skeleton. Arbitrary
connected-subgraph monotonicity fails; the correct comparison holds when
all retained edges are tree-like in the ambient graph. Since
$\mathcal S(G)$ may be disconnected, $\DN_{\LLY}$ is applied only to its
nontrivial connected components.

\begin{proof}[Proof of \cref{thm:skeleton-intro}]
If $E(H)=\varnothing$, then $H=K_1$ and
$\DN_{\LLY}(H)=1\le\DN_{\LLY}(G)$. If
$\DN_{\LLY}(G)=\infty$, the claim is also immediate. We may therefore
assume that $E(H)\ne\varnothing$ and restrict any finite-distortion
admissible $w$ on $G$ to $H$. For $x\sim y\in E(H)$, tree-likeness in
the ambient graph and $d_{x,H}^w\le d_{x,G}^w$,
$d_{y,H}^w\le d_{y,G}^w$ give
\[
\kappa_{\LLY}^{w,H}(x,y)\ge\kappa_{\LLY}^{w,G}(x,y)\ge0
\]
by \cref{lem:tree-like-formula}. Thus the restriction is admissible and
has no larger distortion. Taking infima proves the claim.
\end{proof}

\begin{remark}[Failure of arbitrary subgraph monotonicity]
\label{rem:no-general-subgraph-monotonicity}
The tree-like hypothesis is essential: $ST_{2,2}\subset K_6$, but
\[
\DN_{\LLY}(ST_{2,2})=2,
\qquad
\DN_{\LLY}(K_6)=1,
\]
using \eqref{eq:double-star-DN} and the positive unweighted curvature of
$K_6$. The edges of the double star lie in triangles in the ambient
complete graph and are therefore not tree-like there.
\end{remark}

\begin{corollary}[Degree bound along a tree-like edge]
\label{cor:tree-like-degree-bound}
Let $G$ be a connected locally finite graph. For every $x\sim y\in\Etl(G)$,
\begin{equation}
\label{eq:tree-like-degree-bound}
\DN_{\LLY}(G)
\ge
\sqrt{(d_x-1)(d_y-1)}.
\end{equation}
Consequently, if $\Etl(G)\ne\varnothing$, then
\begin{equation}
\label{eq:tree-like-degree-global}
\DN_{\LLY}(G)
\ge
\sup\left\{\sqrt{(d_x-1)(d_y-1)}:x\sim y\in\Etl(G)\right\}.
\end{equation}
\end{corollary}

\begin{proof}
If $\DN_{\LLY}(G)=\infty$, there is nothing to prove. Otherwise, let
$w$ be any finite-distortion admissible weight and put
$m=\inf_e w_e$ and $M=\sup_e w_e$. On a tree-like edge $xy$,
\eqref{eq:local-nonnegative} gives
\[
M^2\ge w_{xy}^2
\ge(d_x-1)(d_y-1)m^2.
\]
Thus $M/m\ge\sqrt{(d_x-1)(d_y-1)}$. Taking infima over weights and then
the supremum over tree-like edges proves the two claims.
\end{proof}

Beyond these scalar comparisons, the ambient weight itself can be compared pointwise with the canonical weight on every tree component of the skeleton.

\begin{proposition}
\label{thm:canonical-skeleton-domination}
Let $G$ be connected and locally finite, let $w$ be admissible with
$\Dist(w)<\infty$, and let $T$ be a nontrivial tree component of
$\mathcal S(G)$. Put
\begin{equation}
\label{eq:skeleton-component-inf}
m_T(w):=\inf_{e\in E(T)}w_e>0.
\end{equation}
Then $\DN_{\LLY}(T)<\infty$ and, for its canonical weight $w_T^\star$,
\begin{equation}
\label{eq:canonical-skeleton-domination}
w_e\ge m_T(w)\,w_T^\star(e)
\qquad\text{for every }e\in E(T).
\end{equation}
Consequently,
\begin{equation}
\label{eq:component-distortion-domination}
\DN_{\LLY}(T)
\le
\Dist(w|_{E(T)})
\le
\Dist(w).
\end{equation}
\end{proposition}

\begin{proof}
By \cref{thm:skeleton-intro}, the restricted weight $w|_{E(T)}$ is admissible. After normalization by $m_T(w)$, \cref{lem:monotone-iteration} gives a bounded normalized fixed point below it, which is $w_T^\star$ by \cref{thm:tree-fixed-point,thm:infinite-fixed-exhaustion-intro}. Thus \eqref{eq:canonical-skeleton-domination} holds. Since $w|_{E(T)}$ is admissible, \[ \DN_{\LLY}(T)\le \Dist(w|_{E(T)})\le \Dist(w), \] which is \eqref{eq:component-distortion-domination}.
\end{proof}

\begin{remark}
\label{rem:skeleton-global}
For the three-center tree in \cref{ex:three-center} with $m=2,r=1$,
the edgewise bound in \cref{cor:tree-like-degree-bound} gives only $2$, whereas the full component gives
$\DN_{\LLY}(T)=1+\sqrt3$. Thus componentwise fixed-point information is
strictly stronger than the local degree bound.
\end{remark}

\subsection{Structural consequences}
\label{subsec:skeleton-structure}

\begin{corollary}[Topology of tree-like components] \label{cor:finite-distortion-skeleton-structure} Let $G$ be a connected locally finite graph with \[ D:=\DN_{\LLY}(G)<\infty, \] and suppose that $G\ne C_n$ for every $n\ge6$. Then every nontrivial component $T$ of $\mathcal S(G)$ satisfies \[ |B(T)|\le\lceil D\rceil. \] If $T$ is infinite, then it is either a double ray or a one-ended tree obtained from a finite tree by attaching a degree-two ray. \end{corollary} 

\begin{proof}
By \cref{thm:necessary}, every nontrivial component $T$ of
$\mathcal S(G)$ is a tree, while
\cref{thm:skeleton-intro} gives $\DN_{\LLY}(T)\le D$. The assertions now follow from \cref{thm:branch-intro,thm:ends-intro}. 
\end{proof}

\begin{corollary}[Finite high-girth classification]
\label{cor:girth6}
Let $G$ be finite and connected with girth at least $6$. Then
\[
\DN_{\LLY}(G)<\infty
\quad\Longleftrightarrow\quad
G\text{ is a tree or }G=C_n\text{ for some }n\ge6.
\]
Consequently, if $|E|>|V|$ then $\DN_{\LLY}(G)=\infty$, while if $|E|=|V|$ then finite distortion occurs exactly for the cycles $C_n$, $n\ge6$.
\end{corollary}

\begin{proof} Here $\Etl(G)=E$. Thus \cref{thm:necessary} leaves only trees and cycles
$C_n$, $n\ge6$. Trees have finite distortion by \cref{thm:tree-fixed-point},
and the unit weight gives $\DN_{\LLY}(C_n)=1$ by
\eqref{eq:tree-like-curvature}. The edge-count consequences are then
immediate.
\end{proof}

\begin{corollary}[Infinite high-girth classification]
\label{cor:infinite-high-girth}
Let $G$ be a connected locally finite infinite graph with girth at least $6$. Then
\[
\DN_{\LLY}(G)<\infty
\]
if and only if exactly one of the following holds:
\begin{enumerate}[label=\textup{(\roman*)}]
\item $G$ is a double ray;
\item $G$ is obtained from a finite tree by attaching one ray.
\end{enumerate}
\end{corollary}

\begin{proof}
Here $\mathcal S(G)=G$. Thus \cref{thm:necessary,thm:ends-intro} give
both the classification and the converse.
\end{proof}

\enlargethispage{4\baselineskip}
\section*{Acknowledgements}

The author used ChatGPT (OpenAI) for brainstorming, exploratory searches, and suggestions on exposition and LaTeX organization. All mathematical statements and references were independently verified by the author.

\end{document}